\documentclass[12pt]{amsart}
\usepackage{amsmath}
\usepackage{amssymb}
\usepackage{amsthm}
\usepackage{amsaddr}
\usepackage{mathrsfs}
\usepackage[all,cmtip]{xy}
\usepackage{times}
\usepackage{graphicx}
\usepackage{multicol}
\def\ra{\rightarrow}
\def\longra{\longrightarrow}
\def\iff{\Longleftrightarrow}

\newcommand{\nobra}[1]{#1^{\uparrow}}

\makeatletter
\newcommand{\eqnum}{\refstepcounter{equation}\textup{\tagform@{\theequation}}}
\makeatother

\def\dom{{\mathbf d}}
\def\curlyD{{\mathcal D}}

\def\I{{\mathscr I}}

\def\J{{\mathcal J}}

\def\N{{\mathbb N}}
\def\P{{\mathcal P}}

\def\R{{\mathcal R}}
\def\ran{{\mathbf r}}

\def\relrho{\, \rho \,}
\def\varep{\varepsilon}

\def\st{\operatorname{star}}

\def\fim{\operatorname{FIM}}

\def\id{\operatorname{id}}

\def\tr{\operatorname{tr}}

\def\<{\langle}
\def\>{\rangle}
\def\leq{\leqslant}
\def\geq{\geqslant}

\newcommand{\twobar}{/\kern-0.3em/}

\newcommand{\mat}[2]{\left( \begin{array}{cc} 1 & 2 \\ #1 & #2 \end{array} \right)}

\numberwithin{equation}{section}

\newtheorem{theorem}{Theorem}[section]
\newtheorem{prop}[theorem]{Proposition}
\newtheorem{lemma}[theorem]{Lemma}
\newtheorem{cor}[theorem]{Corollary}

\theoremstyle{definition}

\newtheorem{example}[theorem]{Example}
\newtheorem{remark}[theorem]{Remark}

\numberwithin{equation}{section}

\title{Ordered groupoid quotients and congruences on inverse semigroups}

\author{Nouf AlYamani}

\address{Kingdom of Saudi Arabia, Ministry of Higher Education,\\University of Dammam,\\P.O.Box 1982, Dammam 31441, Saudi Arabia.\\naalyamani@uod.edu.sa}

\author{N. D. Gilbert}
\address{
School of Mathematical and Computer Sciences\\
and the Maxwell Institute for the Mathematical Sciences,\\
Heriot-Watt University, \\Edinburgh EH14 4AS, U.K\\N.D.Gilbert@hw.ac.uk}

	\thanks{The authors gratefully  acknowledges the support of a research grant from the University of Dammam.
}

\date{}
\begin{document}

\begin{abstract}
We introduce a preorder on an inverse semigroup $S$ associated to any normal inverse subsemigroup $N$, that
lies between the natural partial order and Green's $\J$--relation.  The corresponding equivalence relation $\simeq_N$ is not necessarily
a congruence on $S$, but the quotient set does inherit a natural ordered groupoid structure.  We show that this construction permits the
factorisation of any inverse semigroup homomorphism into a composition of a quotient map and a star-injective functor, and that
this decomposition implies a classification of congruences on $S$.  We give an application to the congruence and certain normal inverse subsemigroups associate to an inverse monoid presentation.
\end{abstract}



\maketitle

\section*{Introduction}
Let $S$ be an inverse semigroup with semilattice of idempotents $E(S)$.
Recall that the \emph{natural partial order} on $S$ is defined by
\[ s \leq t \; \iff \; \text{there exists} \; e \in E(S) \; \text{with} \; s=et \,.\]
The natural partial order may be characterized in a number of alternative ways, including:
\begin{itemize}
\item there exists $f \in E(S)$ with $s=tf$, 
\item $s=ss^{-1}t$, 
\item $s=ts^{-1}s$.
\end{itemize}
(see \cite[Proposition 5.2.1]{HoBook}).  In this paper, we shall generalize the natural partial order by introducing a preorder $\leq_N$
on $S$ for any normal inverse subsemigroup $N$: the natural partial order then corresponds to the minimal
normal inverse subsemigroup $E(S)$, and at the other extreme, the preorder associated to $S$ itself is the $\J$--preorder.
Symmetrizing the preorder $\leq_N$ yields an equivalence relation $\simeq_N$ (the identity when $N=E(S)$ and the
$\J$--relation when $N=S$).  However, this relation need not be a congruence, and so the set of equivalence classes $S/ \simeq_N$ need
not be an inverse semigroup.

However, we may investigate $\simeq_N$ further by exploiting the relationship between inverse semigroups and ordered groupoids.
An ordered groupoid is a small category in which every morphism is invertible, equipped with a partial order on morphisms.
(The definition is recalled in detail in section \ref{prelims}.)  An inverse semigroup can be considered as an ordered groupoid
in which the identities form a semilattice, and from any such ordered groupoid a corresponding inverse semigroup can be 
constructed.  In the study of inverse semigroups, it is often fruitful to extend the point of view to ordered groupoids, and this is
a major theme of \cite {LwBook}.   We show that the quotient set $S/ \simeq_N$ always inherits a natural ordered groupoid structure.
Moreover, to any homomorphism $\phi : S \ra \Sigma$ of inverse semigroups, we associate its kernel
$K = \{ s \in S : s \phi \in E(\Sigma) \}$, and $\phi$ then factorises as
\[ S \longra S/ \simeq_K \longra \Sigma \]
with the map $S / \simeq_K \ra \Sigma$ a star-injective functor from the ordered groupoid $S / \simeq_K$ to $\Sigma$
(considered as an ordered groupoid).

Now any congruence $\rho$ on an inverse semigroup determines a normal inverse subsemigroup $K$, its \emph{kernel},
which consists of all elements of $S$ that are $\rho$--equivalent to idempotents.  We compare the structures of the 
ordered groupoid $S / \simeq_K$ and the quotient inverse semigroup $S/ \, \rho$, and we show that if $\simeq_K$ is a 
congruence, then it is the minimal congruence with kernel $K$.  We show how to classify congruences by the factorisation
\[ S \longra S / \simeq_K \longra S /  \rho \]
and look at certain congruences and their kernels associated with an inverse monoid presentation, and the relationships
between them.

\section{Ordered groupoids and inverse semigroups}
\label{prelims}
A {\em groupoid} $G$ is a small category in which every morphism
is invertible.  We consider a groupoid as an algebraic structure
following \cite{HiBook}: the elements are the morphisms, and
composition is an associative partial binary operation.
The set of identities in $G$ is denoted $E(G)$, and an element $g \in G$ has domain
$gg^{-1}$ and range $g^{-1}g$.    

An {\em ordered groupoid} $(G,\leq)$ is a groupoid $G$
with a partial order $\leq$ satisfying the following axioms:
\begin{enumerate}
\item[OG1] for all $g,h \in G$, if $g \leq h$ then $g^{-1} \leq h^{-1}$,
\item[OG2] if $g_1 \leq g_2 \,, h_1 \leq h_2$ and if the compositions
$g_1h_1$ and $g_2h_2$ are   defined, then $g_1h_1 \leq g_2h_2$,
\item[OG3] if $g \in G$ and $x$ is an identity of $G$ with $x \leq
g \dom$, there exists a unique element $(x|  g)$, called the {\em restriction}
of $g$ to $x$, such that $(x|g)(x|g)^{-1} =x$ and $(x|g) \leq g$,
\end{enumerate}
As a consequence of [OG3] we also have:
\begin{enumerate}
\item[OG3*] if $g \in G$ and $y$ is an identity of $G$ with $y \leq
g \ran$, there exists a unique element $(g|  y)$, called the {\em corestriction}
of $g$ to $y$, such that $(g|y)^{-1}(g|y)=y$ and $(g| y) \leq g$,
\end{enumerate}
since the corestriction of $g$ to $y$ may be defined as $(y |  g^{-1})^{-1}$.

Let $G$ be an ordered groupoid and let $a,b \in G$.  If 
$a^{-1}a$ and $bb^{-1}$ have a greatest lower bound $\ell \in E(G)$, then we may define the
{\em pseudoproduct} of $a$ and $b$ in $G$ as
$a \otimes b = (a | \ell)(\ell|  b)$,
where the right-hand side is now a composition defined in $G$. As 
Lawson shows in Lemma 4.1.6 of \cite{LwBook}, this is a partially
defined associative operation on $G$. 

If $E(G)$ is a meet semilattice then $G$ is called an \emph{inductive} groupoid. The pseudoproduct  is then everywhere defined and
$(G,\otimes)$ is an inverse semigroup.  On the other hand, given an inverse semigroup $S$ with
semilattice of idempotents $E(S)$, then $S$ is a poset under the natural partial order, and the
restriction of its multiplication to the partial composition
\[ s \cdot t = st \in S \; \text{defined when} \; s^{-1}s=tt^{-1} \]
gives $S$ the structure of an ordered groupoid, with set of identities $E(S)$.  These constructions 
give an isomorphism between the categories of inverse semigroups and inductive groupoids: this is the 
\emph{Ehresmann-Schein-Nambooripad Theorem} \cite[Theorem 4.1.8]{LwBook}.
We call a product $st \in S$ with $s^{-1}s = tt^{-1}$ a \emph{trace product}.
Any product in $S$ can be expressed as a trace product, at the expense of changing the factors, since $st = stt^{-1} \cdot s^{-1}st$.

Let $e \in E(G)$.  Then the {\em star} of $e$ in $G$ is the set
$\st_G(e) = \{ g \in G : gg^{-1} =e \}$. 
A functor
$\phi : G \ra H$ is said to be {\em star-injective} if, for each $e \in E(G)$,
the restriction $\phi : \st_G(e) \ra \st_H(e \phi)$ is injective.
A star-injective functor is also called an \emph{immersion}.  If $G$ is inductive, then $\st_G(e)$ is just the Green
$\R$--class of $e$ in the inverse semigroup $(G,\otimes)$.

\section{Normal inverse subsemigroups and quotients}
An inverse subsemigroup $N$ of an inverse semigroup $S$ is \emph{normal} \cite{Pe} if it is full -- that is, if $E(N)=E(S)$ -- and if, for all
$s \in S$ and $n \in N$, we have $s^{-1}ns \in N$.   A normal inverse subsemigroup $N$ of $S$ determines a relation $\leq_N$ on $S$,
defined using the natural partial order $\leq$ on $S$, as follows:
\begin{equation}
\label{leq_def}
 s \leq_N t \; \iff \; \text{there exist} \; a,b \in N \; \text{such that} \; a \cdot s \cdot b \leq t \,. \end{equation}
Note the requirement that trace products occur here.  We define the relation $\simeq_N$ by symmetrizing $\leq_N$:
\begin{equation}
\label{simeq_def}
 s \simeq_N t \; \iff \; \text{there exist} \; a,b,c,d \in N \; \text{such that} \; a \cdot s \cdot b \leq t \; \text{and} \;
c \cdot t \cdot d \leq s  \end{equation}

\begin{lemma} \leavevmode
\label{leq_props}
\begin{enumerate}
\item The relation $\leq_{E(S)}$ is the natural partial order $\leq$ on $S$.
\item The relation $\leq_S$ is the $\J$--preorder $\preceq_{\J}$ on $S$.
\item If $s \leq t$ in the natural partial order on $S$ then $s \leq_N t$ for any normal inverse subsemigroup $N$ of $S$.
\item $s \leq_N e$ for some $e \in E(S)$ if and only if  $s \in N$.
\item If $s \leq_N s^2$ then $s \in N$.
\item If $s \leq_N t$ then $st^{-1} \in N$.
\item The relation $\leq_N$ is a preorder on $S$ and hence $\simeq_N$
is an equivalence relation on $S$.
\end{enumerate}
\end{lemma}

\begin{proof} 
\begin{enumerate}
\item This is clear, since if $e,f \in E(S)$, a trace product $e \cdot s \cdot f$ is equal to $s$.
\item If $a \cdot s \cdot b \leq t$ then $aa^{-1} \leq tt^{-1}$ and $a^{-1}a=ss^{-1}$.  Hence
$s \preceq_{\J} t$.  On the other hand, if $s \preceq_{\J} t$ then there exists $p \in S$ with $pp^{-1} \leq tt^{-1}$
and $p^{-1}p=ss^{-1}$.  Then
\[ p \cdot s \cdot (s^{-1} \cdot p^{-1} \cdot pp^{-1}t) \leq t \]
and so $s \leq_S t$.
\item Since $N$ is full, $ss^{-1}, s^{-1}s \in N$ and $s = ss^{-1} \cdot s \cdot s^{-1}s \leq t$.
\item Suppose that $a,b \in N$ with $a \cdot s \cdot b \leq e$.  Therefore
$a \cdot s \cdot b = f \leq e$ with $f \in E(S)$, and then $s = a^{-1}a \cdot s \cdot bb^{-1} = a^{-1}fb^{-1} \in N$.
Conversely, if $n \in N$ then $nn^{-1} \cdot n \cdot n^{-1} = nn^{-1}$ and so $n \leq_N nn^{-1}$.
\item We have $a,b \in N$ with $a \cdot s \cdot b \leq s^2$ and so $s \leq a^{-1}s^2b^{-1}$.  Therefore
$s =ss^{-1}a^{-1}s^2b^{-1}$ and it follows that
\[ s^{-1} = s^{-1}ss^{-1} = s^{-1}ss^{-1}a^{-1}s^2b^{-1}s^{-1} = (s^{-1}a^{-1}s)(sb^{-1}s^{-1}) \in N \]
and so $s \in N$.
\item If $a \cdot s \cdot b \leq t$ then $s = a^{-1}a \cdot s \cdot bb^{-1} \leq a^{-1}tb^{-1}$ and so
$st^{-1} \leq a^{-1}(tb^{-1}t^{-1}) \in N$.  Since $N$ is full, we deduce that $st^{-1} \in N$.
\item It is clear that $\leq$ is reflexive.  Suppose that $s,t,u \in S$ and that $s \leq_N t \leq_N u$.
There exist $a,b,p,q \in N$ such that $a \cdot s \cdot b \leq t$ and $p \cdot t \cdot q \leq u$.  Then
$(pa)s(bq) \leq u$, and $(pa)s(bq)$ is the trace product $(pa) \cdot s \cdot(bq)$ since
\begin{align*}
(pa)^{-1}(pa) &= a^{-1}p^{-1}pa = a^{-1}tt^{-1}a = a^{-1}a = ss^{-1},
\end{align*}
and $aa^{-1} \leq tt^{-1}$. Similarly $s^{-1}s = (bq)(bq)^{-1}$.  Therefore $(pa) \cdot s \cdot (bq) \leq u$
and $s \leq_N u$.
\end{enumerate}
\end{proof}

\begin{cor}
\label{posets}
The normal inverse subsemigroup $N$ is determined by the preorder $\leq_N$, and we obtain an order-preserving
embedding of the poset of normal inverse subsemigroups of $S$ into the poset of preorders on $S$ that contain the
natural partial order.
\end{cor}

\begin{proof}
Part (d) of Lemma \ref{leq_props} shows that
\[ N = \{ s \in S : \text{there exists} \; e \in E(S) \; \text{with} \; s \leq_N e \}.\]
\end{proof}

\begin{remark}
Not every preorder containing the natural partial order arises from a normal inverse subsemigroup.  Consider the
symmetric inverse monoid $\I_n$ and define a preorder $\preceq$ by
$\alpha \preceq \beta \; \Longleftrightarrow \dom(\alpha) \subseteq \dom{\beta}$
where $\dom(\gamma)$ is the domain of $\gamma \in \I_n$.  Then $\alpha \preceq \id$ for all $\alpha \in \I_n$,
and so the normal inverse subsemigroup associated to $\preceq$ is $I_n$ itself, but $\preceq$ is not the 
$\J$--preorder on $\I_n$ and so is not equal to $\leq_{\I_n}$.
\end{remark}

We denote the $\simeq$--class of $s \in S$ by $[s]_N$.

\begin{prop} \leavevmode
\label{simeq_props}
\begin{enumerate}
\item If $n \in N$ then $nn^{-1} \simeq_N n \simeq_N n^{-1} \simeq_N n^{-1}n$,
\item The equivalence relation $\simeq_N$ saturates $N$,
\item The equivalence relation $\simeq_N$ determines $N$ as
\[ N = \bigcup_{e \in E(S)} [e]_N \,.\]
\item If $s \simeq_N t$ then $ss^{-1} \simeq_N tt^{-1}, s^{-1}s \simeq_N t^{-1}t$, and $s^{-1} \simeq_N t^{-1}$.
\item The restriction of $\simeq_N$ to $E(S)$ coincides with Green's $\J$--relation $\J_N$ induced on $E(S)=E(N)$,
\item In the case $N=S$ the relation $\simeq_S$ coincides with Green's $\J$--relation on $S$,
\item In the case $N=E(S)$ the relation $\simeq_{E(S)}$ is the trivial relation on $S$.
\end{enumerate}
\end{prop}

\begin{proof} \leavevmode
\begin{enumerate}
\item If $n \in N$ then $n^{-1} \cdot n \cdot n^{-1} = n^{-1}$ and $n \cdot n^{-1} \cdot n = n$, and hence $n \simeq_N n^{-1}$.
Similarly $nn^{-1} \cdot n \cdot n^{-1} = nn^{-1}$ and $nn^{-1} \cdot nn^{-1} \cdot n = n$, whence $n \simeq_N nn^{-1}$.

\item Suppose that $s \in S$ and that for some $n \in N$ we have $s \simeq_N n$.  By part (a) we may assume that $n \in E(S)$:
then there exist $p,q \in N$ such that $p \cdot s \cdot q \leq n$.  Hence for some $e \in E(S)$ we have $p \cdot s \cdot q=e$
and so $s = p^{-1} \cdot e \cdot q^{-1} = p^{-1}q^{-1} \in N$.

\item This follows from parts (a) and (b).

\item Suppose that $s \simeq_N t$, with $a,b,c,d \in N$ as in \eqref{simeq_def}.  Then $bb^{-1}=s^{-1}s$, $b^{-1}b \leq t^{-1}t,
dd^{-1}=t^{-1}t, d^{-1}d \leq s^{-1}s$ and so $s^{-1}s \simeq_N t^{-1}t$.   Similarly $ss^{-1} \simeq tt^{-1}$.  Since $b^{-1} \cdot s^{-1} \cdot a^{-1} \leq t^{-1}$ and 
$d^{-1} \cdot t^{-1} \cdot c^{-1} \leq s^{-1}$, we also have $s^{-1} \simeq t^{-1}$.

\item If $e \simeq_N f$ then there exist $a,b,p,q \in N$ with $a \cdot e \cdot b \leq f$ and $p \cdot f \cdot q \leq e$.
Therefore we have $a^{-1}a=e$ and $aa^{-1} \leq f$, and so $e \leq_{\J_N} f$.  By symmetry $f \leq_{\J_N} e$
and so $e \; \J_N \; f$.  Conversely, if $e \; \J_N \; f$ there exist $m,n \in N$ with $mm^{-1} \leq f$ and $m^{-1}m=e, nn^{-1} \leq e$ 
and $n^{-1}n=f$.  Then $m \cdot e \cdot m^{-1} \leq f$ and $n \cdot f \cdot n^{-1} \leq e$, and so $e \simeq_N f$.

\item This follows from part (b) of Lemma \ref{leq_props}.

\item By Lemma \ref{leq_props}(a), $\leq_{E(S)}$ is the natural partial order , which is of course anti-symmetric.

\end{enumerate}
\end{proof}

However, $\simeq_N$ need not be a congruence on $S$. 

\begin{example} \leavevmode \label{I4egs}
\begin{enumerate}
\item In the symmetric inverse monoid $\I_4$, let $f: \{1 \} \ra \{2 \}$,  let $S$ be the inverse subsemigroup 
\[ S = \{ \id_{\{1,3\}}, \id_{\{1\}}, \id_{\{2\}}, f, f^{-1}, 0 \} \]
of $\I_4$, and let $N=S$.  Then $\id_{\{1\}} = ff^{-1} \simeq_S f^{-1}f = \id_{\{2\}}$.
But $\id_{\{1,3\}} \id_{\{1\}} = \id_{\{1\}}$ is not $\simeq_S$--related to  $\id_{\{1,3\}} \id_{\{2\}} = 0$. In this example,
the poset of $\J$--classes is just a three-element chain and so is a semilattice.
\item Now let $g:  \{3 \} \ra \{4 \}$ in $\I_4$ and let $T$ be the inverse subsemigroup of $\I_4$
generated by $\{ \id_{\{1,3\}}, \id_{\{2,4\}}, f,g \}$.  Here the $\J$--classes do not form a semilattice, and so
$\simeq_T$ is not  a congruence, and the quotient $T / \simeq_T$ is not an inverse semigroup: it is the poset
\begin{center}
\includegraphics[height=3cm]{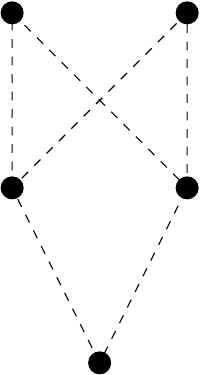}
\end{center}
\end{enumerate}
\end{example}

Following the notation in \cite{AGM}, we shall denote the quotient $S / \simeq_N$ by
$S \twobar N$ and let $\pi : S \ra S \twobar N$ be the quotient map.  Our next  result sets out the ordered groupoid structure on $S \twobar N$: it is a special case of 
\cite[Theorem 3.14]{AGM}, but the description is simpler for quotients of inverse semigroups and seems worth stating
in detail.

\begin{theorem}
\label{is_quot_is_og}
For any inverse semigroup $S$ and normal inverse subsemigroup $N$, the quotient set $S \twobar N$ is an ordered 
groupoid, with the following structure:
\begin{enumerate}
\item the identities are the classes $[e]_N$ where $e \in E(S)$, and a class $[s]_N$ has domain $[ss^{-1}]_N$,
range $[s^{-1}s]_N$, and inverse $[s^{-1}]_N$.
\item as a poset, $E(S \twobar N)$ is isomorphic to $N/\J_N$.
\item  If $s,t \in S$ and $s^{-1}s \simeq_N tt^{-1}$, then there
exists $a \in N$ with $aa^{-1} \leq s^{-1}s$ and $a^{-1}a=tt^{-1}$: the composition of $[s]_N$ and $[t]_N$ is then
defined as $[sat]_N$.  
\item The ordering $\leq_N$ of $\simeq_N$--classes is given by
\[ [s]_N \leq_N [t]_N \; \iff \; \text{there exist} \; a,b \in N \; \text{such that} \; a \cdot s \cdot b \leq t \,. \]
\end{enumerate}
\end{theorem}
 
\begin{proof}
It follows from part (d) of Proposition \ref{simeq_props} that the domain and range of $[s]_N$ and its inverse
$[s]_N^{-1}$ are well-defined.  

Suppose that $s,t \in S$ and $s^{-1}s \simeq_N tt^{-1}$ bu that we choose another
element $z \in N$ with $zz^{-1} \leq s^{-1}s$ and $z^{-1}z=tt^{-1}$.  Then $az^{-1}z=att^{-1}=aa^{-1}a=a$
and $s^{-1}sz = (s^{-1}s)(zz^{-1})z=zz^{-1}z=z$.  Hence
\[ sat = saz^{-1}zt = saz^{-1}s^{-1}szt = (saz^{-1}s^{-1}) \cdot szt \simeq_N szt \]
and so for fixed $s,t$ the $\simeq_N$--class of the element $sat$ does not depend on the choice of $a$.   We denote this class by
$s \between t$.  Suppose that 
$s \simeq_N s_1$ and that we choose $a_1 \in N$ to form $s_1 \between t = [s_1a_1t]_N$.  There exist $u,v \in N$ with 
$u \cdot s_1 \cdot v \leq s$, and so $vv^{-1}=s_1^{-1}s_1 \geq a_1a_1^{-1}$.  Then
\begin{align*}
(v^{-1}a_1)(v^{-1}a_1)^{-1} &= v^{-1}a_1a_1^{-1}v  \leq v^{-1}v \leq s^{-1}s \\
\intertext{and}
(v^{-1}a_1)^{-1}(v^{-1}a_1) &= a_1^{-1}vv^{-1}a_1  = a_1^{-1}a_1 = tt^{-1}.
\end{align*}
Therefore we can use the element $v^{-1}a_1$ to form the class $s \between t = [sv^{-1}a_1t]_N$.  Now 
$sv^{-1}s_1^{-1} = (sv^{-1}v)(v^{-1}s_1^{-1} = (us_1v)(v^{-1}s_1^{-1} = u$ and since $s^{-1}s \geq v^{-1}v$, we have
\[ sv^{-1}a_1t = sv^{-1}vv^{-1}a_1t = sv^{-1}s_1^{-1}s_1a_1t = u \cdot (s_1at) \simeq_N s_1a_1t \]
and so $s \between t \simeq_N s_1 \between t$.  Similarly, $s \between t$ does not depend on the choice of the element
$t$ within its $\simeq_N$-- class, and the product $[s]_N \cdot [t]_N = [sat]_N$ is well-defined.  

Now the relation $s^{-1}s \simeq_N tt^{-1}$ furnishes not only $a \in N$ with $aa^{-1} \leq s^{-1}s$ and $a^{-1}a=tt^{-1}$
but also $p \in N$ with $pp^{-1} \leq s^{-1}s$ and $pp^{-1}=tt^{-1}$.  Consider $saps^{-1} \in N$: we have
\begin{align*}
(saps^{-1})(saps^{-1})^{-1} &= saps^{-1}sp^{-1}a^{-1}s^{-1} \leq sapp^{-1}a^{-1}s^{-1}= (sat)(sat)^{-1}\\
\intertext{and}
(saps^{-1})^{-1}(saps^{-1}) &= sp^{-1}a^{-1}s^{-1}saps^{-1}\\
&= sp^{-1}a^{-1}aps^{-1} \\
&= sp^{-1}ps^{-1} = ss^{-1}
\end{align*}
and, since  $(sat)(sat)^{-1} \leq ss^{-1}$, we have $ss^{-1} \simeq_N (sat)(sat)^{-1}$.  Therefore,
$[sat]_N$ has domain $[ss^{-1}]_N$ and range $[t^{-1}t]_N$, and the composition $[s]_N \cdot [t]_N = [sat]_N$
does give a groupoid structure on $S \twobar N$.

Now by Lemma \ref{leq_props}, $\leq_N$ induces the given partial order on the $\simeq_N$--classes, and it remains to show that 
this partial order makes $S \twobar N$ into an ordered groupoid. 

Now if $a \cdot s \cdot b \leq t$ then $b^{-1} \cdot s^{-1} \cdot a^{-1} \leq t^{-1}$ and so $[s]_N \leq_N [t]_N$ implies that 
$[s]_N^{-1} \leq_n [t]_N^{-1}$.  Now suppose that $[s_1]_N \leq [s]_N, [t_1]_N \leq [t]_N$ and that the compositions
$[s]_N \cdot [t]_N$ and $[s_1] \cdot [t_1]$ exist.  There exist $m,n,u,v \in N$ such that $m \cdot s_1 \cdot m \leq s$
and $u \cdot t_1 \cdot v \leq t$: since $[m \cdot s_1 \cdot n]_N = [s_1]_N$ and $[u \cdot t_1 \cdot v]_N = [t_1]_N$
we may as well assume that $s_1 \leq s$ and $t_1 \leq t$.

We now have $a,p \in N$ with $aa^{-1} \leq s^{-1}s, a^{-1}a = tt^{-1}$ and $b,q \in N$ with
$bb^{-1} \leq s_1^{-1}s_1, b^{-1}b=t_1t_1^{-1}, qq^{-1} \leq t_1t_1^{-1}$ and $q^{-1}q = s_1^{-1}s_1$. 
Now
\[ s_1bt_1 \simeq_N sab^{-1}s^{-1} \cdot s_1bt_1 =  sab^{-1}s_1^{-1}s_1bt_1\leq sat_1 \leq sat \]
and so $[s_1]_N \cdot [t_1]_N = [s_1bt_1]_N \leq [sat]_N = [s]_N \cdot [b]_N$.

Finally we suppose that $[n]_N \leq_N [ss^{-1}]_N$ for some $n \in N$ and $s \in S$.  By part (a) of Proposition \ref{simeq_props}
we may  replace $n$ by$e=nn^{-1}$.  Then there exist $a,b \in N$ with $a \cdot e \cdot b \leq ss^{-1}$ and so
$aa^{-1} \leq ss^{-1}$ and $a^{-1}a=e$.  Then the class $[aa^{-1}s]_N$ has domain $[aa^{-1}]_N = [e]_N$
and $[aa^{-1}s]_N \leq_N [s]_N$.  We wish to show that $[aa^{-1}s]_N$ is the unique $\simeq_N$--class
with these properties.

Suppose that $[k]_N \leq_N [s]_N$ and that $[kk^{-1}]_N=[e]_N$.  There exist $u,v \in N$ with $u \cdot k \cdot v \leq s$
and $b,q \in N$ with $bb^{-1} \leq kk^{-1}, b^{-1}b=e, qq^{-1} \leq e$, and $q^{-1}q = kk^{-1}$.  We have $a \in N$
as in the previous paragraph.  Then $kvs^{-1}=u^{-1}$ and $u^{-1}uq^{-1}=q^{-1}$: hence
\begin{align*}
 k &\simeq_N k \cdot vs^{-1}uq^{-1}a^{-1}s =q^{-1}a^{-1}s = q^{-1}a^{-1}aa^{-1}s \\
&= q^{-1}a^{-1} \cdot aqq^{-1}a^{-1}s \simeq_N = aqq^{-1}a^{-1}s \leq aa^{-1}s \end{align*}
and so $[k]_N \leq_N [aa^{-1}s]_N$.  By symmetry, they are equal. 
\end{proof}

\begin{cor}{\cite[Theorem 4.15]{AGM}}
\label{hom_factors}
Given a homomorphism $\phi : S \ra \Sigma$ of inverse semigroups, let $K = \{ s \in S : x \phi \in E(\Sigma) \}$.  Then
$K$ is a normal inverse subsemigroup of $S$, and $\phi$ factorises as a composition
\[ S \stackrel{\pi}{\longrightarrow} S \twobar K \stackrel{\kappa}{\longrightarrow} \Sigma \]
where $\kappa$, defined by  $[s]_K \kappa = s \phi$, is a star-injective functor.
\end{cor}

\begin{proof}
The map $\kappa$ is well-defined, since if $[s]_K = [s']_K$ then, for some $a,b \in K$ we have
$a \cdot s \cdot b \leq s'$.  Now $a \phi = (a^{-1}a) \phi = (ss^{-1})\phi$ and similarly, $b \phi = (s^{-1}s)\phi)$.
It follows that $(a \cdot s \cdot b) \phi = s \phi \leq s' \phi$.  By symmetry, $s' \phi \leq s \phi$.

To show that $\kappa$ is a functor, suppose that  $[s]_K$ and $[t]_K$ are composable in $S \twobar K$.  Then there exist 
$a,p \in K$ with 
\[ aa^{-1} \leq s^{-1}s,
a^{-1}a = tt^{-1}, pp^{-1} \leq tt^{-1}, p^{-1}p = s^{-1}s\]
 and the composition of $[s]_K$ and $[t]_K$ is defined, as in Theorem \ref{is_quot_is_og}, by 
$[s]_K \cdot [t]_K = [sat]_K$.  Then
\begin{align*}
([sat]_K) \kappa &= (sat)\phi  = (s \phi)(a \phi)(t \phi) \\
&= (s \phi) (a^{-1}a)\phi(t\phi) \; (\text{since} \; a \in K)\\
&= (s\phi)(tt^{-1})\phi (t\phi) = (s\phi)(t\phi) \\
\end{align*}
Now $(s\phi)(t\phi)$ is a trace product $(s\phi) \cdot (t\phi)$ since 
\begin{align*}
(s \phi)^{-1}(s\phi) &=(s^{-1}s)\phi  = (p^{-1}p) \phi \\
&= (pp^{-1}) \phi \; (\text{since} \; p \in K)\\
& \leq (tt^{-1}) \phi  = (t \phi)(t \phi)^{-1}
\end{align*}
and similarly
\begin{align*}
(t \phi)(t \phi)^{-1} &=(tt^{-1})\phi = (a^{-1}a)\phi \\
&= (aa^{-1})\phi  \; (\text{since} \; a \in K)\\
& \leq (s^{-1}s)\phi  = (s \phi)^{-1}(s \phi).
\end{align*}
Therefore $(s\phi)(t\phi)$ is a trace product defined in the inductive groupoid $(\Sigma,\cdot)$ and $\kappa$ is a functor.

To show that $\kappa$ is star-injective, suppose that for some $u,v \in S$ we have $[uu^{-1}]_K = [vv^{-1}]_K$ and $u \phi = v \phi$.  We claim that
$u \simeq_K v$.  By symmetry, it is sufficient to show that $u \leq_K v$.  Now by part (e)
Proposition \ref{simeq_props} there exist $a,b \in K$ with $aa^{-1} \leq uu^{-1}, a^{-1}a = vv^{-1}, bb^{-1} \leq vv^{-1}$
and $b^{-1}b=uu^{-1}$.  Then
\[ b \cdot u \cdot u^{-1}b^{-1}v = (buu^{-1}b^{-1})v \leq v \]
and 
\begin{align*}
(u^{-1}b^{-1}v) \phi &= (u^{-1}) \phi (b^{-1}b) \phi v \phi \; \text{(since $b \in K$)} \\
&= (u^{-1}b^{-1}b) \phi v \phi \\
&= (u^{-1} \phi) v \phi = (u \phi)^{-1} v \phi \in E(\Sigma) \; \text{since $u \phi = v \phi$} \,,
\end{align*}
and so $u^{-1}b^{-1}v \in K$ and $u \leq_k v$ as required.
\end{proof}

\begin{cor}
\label{unique_factor}
The factorization of $\phi : S \ra \Sigma$ is unique, in the sense that if $\phi$ also factorizes as
$S \ra S \twobar N  \stackrel{\nu}{\ra} \Sigma$
with $\nu$ a star-injective functor, then $N=K$ (and hence $\nu = \kappa$.)
\end{cor}

\begin{proof}
If $n \in N$ then by part (a) of Proposition \ref{simeq_props}, we have $n \simeq_N nn^{-1}$ and so $n \phi \in E(\Sigma)$.
Hence $N \subseteq K$.

Now if $k \in K$ then $k \phi \in E(\Sigma)$, and since $\nu$ is star-injective, then $[k]_n$ is an identity in $S \twobar N$
and so, for some $e \in E(S)$ we have $k \simeq_N e$.  Then there exists $a,b \in N$ such that $a \cdot k \cdot b \leq e$,
and so $a \cdot k \cdot  b = f \in E(S)$.  But $a^{-1}a = kk^{-1}$ and $bb^{-1}=k^{-1}k$, so that
\[ k = (kk^{-1})k(k^{-1}k) = a^{-1}akbb^{-1} = a^{-1}fb^{-1} \in N \,.\]
Hence $K \subseteq N$ and so $N=K$.
\end{proof}

We note that this factorisation of an inverse semigroup homomorphism requires the use of an intermediate ordered groupoid.
We shall apply it to the study of congruences in section \ref{congruences}.
For the further study of inverse semigroups, it is clearly of interest to know when we
can form a quotient \emph{inverse semigroup} $S \twobar N$.  Since an inverse semigroup is equivalent to an ordered groupoid
in which the poset of identities is a semilattice,  we have the following.

\begin{prop}
\label{quot_is1}
Let $S$ be an inverse semigroup and $N$ a normal inverse subsemigroup of $S$.  Then 
the quotient ordered groupoid $S \twobar N$ is an inverse semigroup if and only if the poset of
$\J_N$--classes of $S$ is a semilattice.  (This is certainly the case if $\J_N$ is a congruence on $E(N)$.)
\end{prop}

\begin{example}
In the symmetric inverse monoid $\I_n$, let $N$ be the subset of non-permutations together with the identity map $\id$.
Then $N$ is a normal inverse subsemigroup.  Since only $\id \in N$ can form trace products with permutations, 
$\leq_N$ restricts to the identity on the symmetric group $S_n \subset \I_n$.  Moreover, for any $\id \ne \nu \in N$ and $\sigma \in S_n$
we have $\nu^{-1} \cdot \nu \cdot \sigma|_{\ran(\nu)} \leq \sigma$ so that $\nu \leq_N \sigma$.  On the elements of $N$,
the relation $\simeq_N$ is equal to the $\curlyD$ (and $\J$) relation, and so for non-identity $\alpha, \beta \in N$ we have
$\alpha \simeq_N \beta \iff |\dom(\alpha)|=|\dom(\beta)|$.  It follows that $\I_n \twobar S_n$ consists of the group
$S_n$ as a set of $n!$ maximal elements, and a chain $e_{n-1} > e_{n-2} > \dotsb > e_1 > e_0$ of identitites corresponding to the 
cardinalities of non-identity elements of $N$.
\end{example}

\begin{example}{\bf Polycyclic and gauge monoids.}
Let $A = \{ a_1,a_2, \dotsc , a_n \}$ with $n \geq 1$.  The polycyclic monoid $P_n$ (introduced in \cite{NP}) is the inverse hull of $A^*$:
its underlying set is $(A^* \times A^*) \cup \{ 0 \}$ and the multiplication of non-zero elements is given by:
\begin{align*}
(s,t)(u,v) = 
\begin{cases}
(s,pv) &\; \text{if} \; t=pu \; \text{for some} \; p \in A^* \,, \\
(ps,v) & \; \text{if} \; u=pt \; \text{for some} \; p \in A^* \,, \\
0 & \; \text{otherwise.}
\end{cases}
\end{align*}
\end{example}
The semilattice of idempotents is
\[ E(P_n) = \{ (p,p) : p \in A^* \} \cup \{ 0 \} \]
and the natural partial order between non-zero elements is given by $(u,v) \leq (s,t)$ if and only if
$u=ps, v=pt$ for some $p \in A^*$.

Full inverse subsemigroups of $P_n$ have the form $Q \cup \{ 0 \}$ where $Q$ is a left congruence on $A^*$: see
\cite[Theorem 3.3]{Lw1}, with a change to left congruence required by our differing conventions.  
Meakin and Sapir \cite{MeaSap} established the first correspondence of this kind, showing that the lattice of congruences on $A^*$
is isomorphic to the lattice of \emph{positively self-conjugate} submonoids of $P_n$, where an inverse submonoid $R$ is
positively self-conjugate if $(w,1)R(1,w) \subseteq R$ for every $w \in A^*$.

\begin{lemma}
\label{normal_in_Pn}
A full inverse semigroup $N = Q \cup \{ 0 \}$ of $P_n$ is normal if and only if $Q$ is a right cancellative congruence on $A^*$.
\end{lemma}

\begin{proof}
By \cite[Theorem 3.3]{Lw1} $Q$ must be a left congruence on $A^*$.
Suppose that $(q_1,q_2) \in Q$ and that $(h_1,h_2), (k_1,k_2) \leq (w_1,w_2) \in P_n$.  Then $h_i = uw_i$ and
$k_i = vw_i$ for $i=1,2$ and some $u,v \in A^*$.  Then $N$ is normal if and only 
\begin{align*}
(h_1,h_2)^{-1}(q_1,q_2)(k_1,k_2) &= (uw_2,uw_1)(q_1,q_2)(vw_1,vw_2) \\
& = (uw_2,vw_2) \in Q
\end{align*}
where $q_1 = uw_1$ and $q_2=vw_1$.  Therefore $N$ is normal if and only if, for all $w_1,w_2 \in A^*$, we have
that $(uw_1,vw_1) \in Q$ implies that $(uw_2,vw_2) \in Q$ and this is equivalent to $Q$ being a right cancellative
two-sided congruence on $A^*$.
\end{proof}

The \emph{gauge inverse monoid} $G_n$ is defined by 
\[ G_n = \{ (s,t) : |s|=|t| \} \cup \{ 0 \} \,.\]
It was introduced in \cite{JoLw}, and by Lemma \ref{normal_in_Pn} it is a normal inverse submonoid of
$P_n$.  By  \cite[Lemma 3.4]{JoLw} Green's relations $\curlyD$ and $\J$ coincide in $G_n$, and clearly
$(s,t)$ and $(u,v)$ are $\curlyD$--related in $G_n$ if and only if $|s|=|t|=|u|=|v|$.  Thus the non-zero $\J$--classes
are indexed by the non-negative integers and the $\J$--order on them is trivial.  The $\J$--class of $0$ is minimal and $E(P_n)/\J_{G_n}$
is a semilattice, and so $P_n \twobar G_n$ is an inverse semigroup.  To identify it, we note that
$(u,v) \leq_{G_n} (s,t)$ if and only if there exist $h,k  \in A^*$ such that $|h|=|u|, |k|=|v|$ and
\[ (h,k) = (h,u)(u,v)(v,k) \leq (s,t) \,. \]
It follows that $h=ps$ and $k=pt$ for some $p \in A^*$ and so
\[ (u,v) \leq_{G_n} (s,t) \iff \; \text{there exists} \; p \in A^* \; \text{such that} \; |u|-|s|=|v|-|t|=|p| \geq 0 \,.\]
The relation $\simeq_{G_n}$ is then given by
\[ (u,v) \simeq_{G_n} (s,t) \iff |u|=|s| \; \text{and} \; |v|=|t| \]
and the $\simeq_{G_n}$ classes of the non-zero elements in $P_n$ are thus parametrized by pairs of non-negative integers.
Now in $P_n \twobar G_n$ we have
\[ [(u,v)]_{G_n} [(s,t)]_{G_n} = [(u,v)(v,s)(s,t)]_{G_n} = [(u,t)]_{G_n} \]
and so  $P_n \twobar G_n$ is isomorphic to the Brandt semigroup on the set of non-negative integers.

\section{Congruences and kernels}
\label{congruences}
Let $\rho$ be a relation on an inverse semigroup $S$.  Following \cite{Pe}, the \emph{trace} $\tr(\rho)$ of $\rho$ is its restriction to $E(S)$, and the \emph{kernel} $\ker \rho$ is the set
\[ \ker \rho = \{ s \in S : s \relrho e \; \text{for some} \; e \in E(S) \}. \]

\begin{prop}
\label{ker_tr_simeq}
The kernel of the relation $\simeq_N$ is $N$ and its trace is Green's relation $\J_N$ on $E(N)=E(S)$.
\end{prop}

\begin{proof}
This follows from Lemma \ref{leq_props}(c), Proposition \ref{simeq_props}(a) and (d).
\end{proof}

Recall from \cite{Pe} that a congruence $\rho$ on the semilattice of idempotents $E(S)$ of $S$ is \emph{normal} if,
for all $s \in S$, $e \relrho f$ implies that $s^{-1}es \relrho s^{-1}fs$.  Then \cite[Definition 4.2]{Pe} a \emph{congruence pair} 
$(K,\nu)$ on $S$ consists of a normal inverse semigroup
$K$ of $S$ and a normal congruence $\nu$ on $E(S)$ such that 
\begin{itemize}
\item[~] \eqnum\label{conpair1} \; if $e \in E(S)$ and $s \in S$ satisfy $se \in K$ and $s^{-1}s \, \nu \, e$ then $s \in K$,
\item[~]  \eqnum\label{conpair2} \; if $u \in K$ then $uu^{-1} \, \nu \, u^{-1}u$.
\end{itemize}
For any \emph{congruence} $\rho$, its kernel and trace form a congruence pair.  Conversely, given a congruence pair $(K,\nu)$ the
relation $\rho_{(K,\nu)}$ defined by
\[  \eqnum\label{conpair3} s \, \rho_{(K,\nu)} \, t \; \iff \; st^{-1} \in K \; \text{and} \; s^{-1}s \, \nu \, t^{-1}t \]
is a congruence with kernel $K$ and trace $\nu$.  This correspondence is the basis of the characterization of congruences in
\cite[Theorem 4.4]{Pe}.  The lattice of all congruences on both regular and inverse semigroups was earlier studied by
Reilly and Scheiblich \cite{RS}.

If $\rho$ is a congruence on $S$, let $\rho(s)$ be the class of $s \in S$ and  let $\rho_* : S \ra S/ \rho$ be the quotient map, $s \mapsto \rho(s)$.  Now $S/ \rho$ is an inverse semigroup and so
is an inductive groupoid with its trace product.  If $K = \ker \rho$ then $\simeq_K$ is a relation on $S$, and as in \cite{AGM} we have the
quotient map $\pi : S \ra S \twobar K$, $s \mapsto [s]_K$ where $S \twobar K$ is an ordered groupoid.  Applying Corollary \ref{hom_factors} to the  homomorphism $S \ra S/ \rho$ we obtain:

\begin{prop} 
\label{functor}
If $K$ is the kernel of the congruence $\rho$ on an inverse semigroup $S$ then $s \simeq_K t$ implies that $s \relrho t$, and the
induced mapping $\kappa : S \twobar K \ra S / \rho$ carrying $[s]_K \mapsto \rho(s)$ is a surjective star-injective functor.  
\end{prop}

The converse of proposition \ref{functor} is the following.

\begin{theorem}
\label{cong_pair}
Let $N$ be a normal inverse subsemigroup of $S$ and let $\psi : S \twobar N \ra Q$ be a surjective, star-injective functor to
an inverse semigroup $Q$.  Let $\nu$ be the congruence on $E(S)$ determined by the composition
\[ E(S) \ra S \twobar N \ra E(Q) \,.\]
Then $(N,\nu)$ is a congruence pair, and the associated congruence $\rho_{(N,\nu)}$ on $S$ is that determined by the composition
\[ \phi: S \stackrel{\pi}{\longra} S \twobar N \stackrel{\psi}{\longra} Q \,.\]
\end{theorem}

\begin{proof}
It is clear that $\nu$ is a normal congruence on $E(S)$.  Suppose that $e \in E(S), s \in S, se \in N$ and $(s^{-1}s) \phi = e \phi$.
Then 
\[ s \phi = (ss^{-1}s) \phi = (s \phi)(s^{-1}s) \phi = (s \phi)(e \phi) = (se) \phi \]
and $(se) \phi \in E(Q)$ since $se \in N$.  Hence $s \phi = s \pi \nu \in E(Q)$.  Since $\nu$ is star-injective, then $s \pi \in E(S \twobar N)$
and so $s \in N$ by part (c) of Proposition \ref{simeq_props}. 

Now if $u \in N$ then $u \phi = u \pi \psi \in E(Q)$ and so
\[ u \phi = u^{-1} \phi = (uu^{-1}) \phi  = (u^{-1}u) \phi \]
and therefore $uu^{-1} \, \nu \, u^{-1}u$.  This confirms that $(N,\nu)$ is a congruence pair.

If $s,t \in S$ and $s \phi = t \phi$ then $(st^{-1})\phi = (st^{-1}) \pi \psi \in E(Q)$.  Since $\psi$ is star-injective, then
$(st^{-1})\pi \in E(S \twobar N)$ and again $s \in N$ by part (c) of Proposition \ref{simeq_props}.  Moreover,
$(s^{-1}s) \phi = (s \phi)^{-1}(s \phi) = (t \phi)^{-1}(t \phi) = (t^{-1}t)\phi$ and so $s^{-1}s \, \nu \, t^{-1}t$.
Therefore $s \, \rho_{(N,\nu)} \, t$.  Conversely, if $s \, \rho_{(N,\nu)} \, t$ then $st^{-1} \in N$ and $s^{-1}s \, \nu \, t^{-1}t$.
Then $(st^{-1})\phi = (st^{-1}) \pi \psi \in E(Q)$, since $(st^{-1}) \pi \in E(S \twobar N)$, and so
\[ t \phi \geq (st^{-1})\phi (t \phi) = (st^{-1}t) \phi = (ss^{-1}s)\phi = s \phi \,,\]
and by symmetry, $t \phi = s \phi$.
\end{proof}

Howie \cite[Exercise 5.11.16]{HoBook} defines a full inverse semigroup $N$ of an inverse semigroup $S$ to have the \emph{kernel
property} if, whenever $s,t \in S$ with $st \in N$ and $n \in N$ then $snt \in N$.  A full inverse subsemigroup with the
kernel property is called \emph{normal} in \cite{Gr}.  It is  easy to see that an inverse subsemigroup wth the kernel property
is normal in the sense of \cite{Pe} (the sense used in this paper), and that the kernel of any congruence has the
kernel property.  Moreover, an inverse subsemigroup with the kernel property is the kernel of its syntactic congruence, and so:

\begin{theorem}[\cite{Gr}, Theorem 3.3]
\label{ker_prop_iff_kernel}
An inverse subsemigroup $N$ of an inverse semigroup $S$ is the kernel of a congruence if and only if it is full and
has the kernel property.
\end{theorem}

Hence, if $\simeq_N$ is a congruence, $N$ must have the kernel property.

\begin{theorem}
\label{simeq_is_minl}
If $N$ is a full inverse subsemigroup of an inverse semigroup $S$ and $N$ has the kernel property, then $\simeq_N$ is a 
congruence on $S$ if and only if $\J_N$ is a normal congruence on $E(S)$, and  $\simeq_N$ is then the
minimal congruence on $S$ with kernel $N$.
\end{theorem}

\begin{proof}
By Proposition \ref{ker_tr_simeq}, the relation $\simeq_N$ is a congruence on $S$ if and only $(N, \J_N)$ is a
congruence pair.  Now $N$ is normal, and \eqref{conpair2} holds by part (a) of Proposition \ref{simeq_props}

For \eqref{conpair1}, suppose that $se \in N$ and that $s^{-1}s \simeq_N e$.  Then there exists $a \in N$ with
$aa^{-1}  \leq e$ and $a^{-1}a = s^{-1}s$.  Now
\[ se = ss^{-1}se = sess^{-1} \geq saa^{-1}s^{-1} \]
and so $saa^{-1}s^{-1}a \in N$.  Now by two applications of the kernel property using $a^{-1}a \in N$,
\[ sa^{-1}(aa^{-1})as^{-1}s = sa^{-1}as^{-1}s = ss^{-1}ss^{-1}s = s \in N \,.\]
Hence \eqref{conpair1} holds, and therefore $(N,\J_N)$ is a congruence pair if and only if $\J_N$ is a normal
congruence on $E(S)$.

Now if $\rho$ is a congruence with kernel $N$ we have, for $s,t \in S$,
\[ s \simeq_N t \Longrightarrow s \relrho t \,,\]
and so $\simeq_N$ is minimal.
\end{proof}

\subsection{Idempotent separating congruences}
\label{id_sep_cong}
A congruence $\rho$ on an inverse semigroup $S$ is \emph{idempotent separating} if its trace is the identity
relation on $E(S)$.  The classification of congruences by congruence pairs \cite[Theorem 4.4]{Pe} shows that
an idempotent separating congruence is entirely determined by its kernel $K$, a normal inverse subsemigroup of $S$
which, by \eqref{conpair2},  must also satisfy the property that for all $a \in K$, $aa^{-1}=a^{-1}a$.  (Hence $K$ is
a \emph{Clifford} inverse semigroup).  The congruence $\rho$ is then defined, according to  \eqref{conpair3}, by:
\[  \eqnum\label{id_sep_from_kernel} \; s \, \rho \, t \; \iff \; st^{-1} \in K \; \text{and} \; s^{-1}s = t^{-1}t \,.\]

\begin{prop}
\label{idpt-sep}
If $\rho$ is an idempotent-separating congruence on $S$ with kernel $K$ then the relations $\rho$ and $\simeq_K$ are equal, and
so $\kappa : S \twobar K \ra S/\rho$ is an isomorphism of inverse semigroups.
\end{prop}

\begin{proof}
If $s \relrho t$ then 
$st^{-1} \in K$ and since $s^{-1}s = t^{-1}t$, we see that $st^{-1}$ is a trace product  in $(S,\cdot)$. Hence $s \cdot t^{-1} \cdot t$ is also a trace product in $S$, and $st^{-1}t \leq s$.  Since $st^{-1}$ in $K$, this shows that $[t]_K \leq [s]_K$.  By symmetry,
they are equal (or we can repeat the argument using $ts^{-1} \in K$ and $ts^{-1}s \leq t$).  So if $\rho$ is idempotent-separating then $s \relrho t$ implies that $s \simeq_K t$, and Proposition \ref{functor} gives the reverse
implication.
\end{proof}

\begin{remark}
The converse of this result is not true: see Example \ref{simple_x_grp} below.
\end{remark}

\subsection{Closed inverse subsemigroups}
\label{cliss}
For a subset $A$ of an inverse semigroup $S$, we denote by $\nobra{A}$ the smallest closed subset of
$S$ containing $A$.  If $A$ is an inverse subsemigroup of $S$, then so is $\nobra{A}$.

Let $N$ be a \emph{closed} inverse subsemigroup of $S$: so if $n \in N$ and $n \leq s$ then $s \in N$.
The relation $a \equiv_N b \; \iff \; ab^{-1} \in N$ is then an equivalence relation on the subset
$\st_S(N) = \{ s \in S : ss^{-1} \in N \}$ and the equivalence classes are the \emph{cosets} of $N$.  This notion of 
coset was introduced by Schein \cite{Sch}.  If $N$ is normal, then $\st_S(N)=S$ and $\equiv_N$ is an
equivalence relation on $S$ and it is easy to see that it is then also a congruence, with kernel $N$.  If $s,t \in S$ and there exists $e \in E(S)$
with $es=et$ then $est^{-1}e \in E(S) \subseteq N$ and, since $N$ is closed, $st^{-1} \in N$.  It follows that $\equiv_N$
contains the minimal group congruence $\sigma$ on $S$, and if $\sigma_* : S \ra S / \sigma$ then $S / \! \equiv_N$ is
isomorphic to the quotient group $(S / \sigma)/N \sigma_*$.  

The relation $\simeq_N$ is finer than $\equiv_N$:

\begin{prop}
\label{quots_by_N}
Let $N$ be a closed normal inverse subsemigroup of $S$.  Then for all $s,t \in S$, $s \simeq_N t$ implies that
$s \equiv_N t$, and $s \equiv_N t$ if and only if $st^{-1}t \simeq_N ts^{-1}s$.
\end{prop}

\begin{proof}
If $s \simeq_N t$ then $s \equiv_N t$ by part (f) of Lemma \ref{leq_props}.
Now if $st^{-1} \in N$ we have $st^{-1} \cdot ts^{-1}s = st^{-1}t$ and $ts^{-1} \cdot st^{-1}t = ts^{-1}s$,
and so $st^{-1}t \simeq_N ts^{-1}s$.  Conversely, if $st^{-1}t \simeq_N ts^{-1}s$, then by the first part
$st^{-1} = (st^{-1}t)(s^{-1}st^{-1}) \in N$ and so $s \equiv_N t$. 
 \end{proof}

\begin{example}
\label{simple_x_grp}
Let $T$ be any inverse semigroup and let $G$ be a group: set $S = T \times G$.  We identify $T$ with  $T_1 = \{ (t,1_G) : t \in T \}$,
which is a closed, normal, inverse subsemigroup of $S$.  Then $(u,g) \equiv_T (v,h)$ if and only $g=h$, so that
$S/ \! \equiv_T$ is isomorphic to $G$.  However, by part (f) of Proposition \ref{simeq_props}, we have
\[ (u,g) \simeq_T (v,h) \iff u \J t \; \text{and} \; g=h \]
and so $S \twobar T \cong (T/ \J) \times G$.  Hence $\kappa: S \twobar T \ra S / \equiv_T$ is an isomorphism if and only if
$T$ is simple.

If $T$ is not a group, then $\equiv_T$ is not idempotent separating, and so the converse of Proposition \ref{idpt-sep} is
not true in general.
\end{example}

\section{Inverse monoid presentations}
Let $\P = \langle X:R \rangle$ be  a presentation of the inverse monoid $M$.  We assume that $R$ consists of a set
of pairs $(\ell,r)$ with $\ell,r \in \fim(X)$, the free inverse monoid on $X$.  The pairs in $R$ generate a congruence 
$\simeq_{\P}$ on $\fim(X)$ with $M$ isomorphic to the quotient $\fim(X) / \simeq_{\P}$.  We let
$\pi : \fim(X) \ra M$ denote the quotient map.

Let $K(\P)$ be the kernel of $\simeq_{\P}$.  We note that $K(\P)$ is the image in $\fim(X)$ of the idempotent problem
(see \cite{GiNH}) of $\P$ in $(X \sqcup X^{-1})^\ast$.  By Theorem \ref{ker_prop_iff_kernel} $K(\P)$ is a full
inverse subsemigroup of $\fim(X)$ with the kernel property, and hence is normal.

\begin{prop}
\label{comp_KandN}
Let $N(\P)$ be the smallest normal inverse subsemigroup of $\fim(X)$ containing the set
\[ Q(R) = \{ \ell^{-1}r, \ell r^{-1} : (\ell,r) \in R \} \,. \] 
Then 
\[ N(\P) \subseteq K(\P) \subseteq \nobra{N(\P)} \,.\]
\end{prop}

\begin{proof}
Elements of $N(\P)$ are products of conjugates of elements of $Q(R)$ and their inverses and idempotents in $\fim(X)$
Since each element of $Q(R)$ is mapped by $\pi$ to an idempotent of $M$ we therefore have $N(\P) \subseteq K(\P)$.  

Suppose now that $u \simeq_{\P} v$: then there exists $u=u_0, u_1, \dotsc , u_{k-1}, u_k=v$ such that, for all $i$ with
$0 \leq i \leq k-1$, there exist $p_i, q_i \in \fim(X)$ such that $u_i = p_i \ell q_i$ and $u_{i+1} = p_irq_i$, or vice versa.  Assume the former: then
\[ p_irq_i \geq p_ir \ell^{-1} p_i^{-1}p_i \ell q_i \]
with $p_ir \ell^{-1} p_i^{-1} \in N$.  Hence $u_{i+1} \geq n_iu_i$ and so, for some $n \in N$, we have $v \geq nu$.
Hence if for some $e \in E(\fim(X))$ we have $e \simeq _{\P} v$, that is if $v \in K(\P)$, then $v \geq ne$ with $ne \in N$
and so $v \in \nobra{N}$.
\end{proof}

\begin{prop}
\label{Eunit_has_closed_ker}
The monoid $M$ is $E$--unitary if and only if $K(\P)$ is a closed inverse submonoid of $\fim(X)$, in which
case $K(\P) = \nobra{N(\P)}$.
\end{prop}

\begin{proof}
Suppose that $M$ is $E$--unitary, that $u \geq v$ in $\fim(X)$ and that $v \in K(\P)$.  Then $v \pi \in E(M)$ and since 
$u \pi \geq v \pi$ we have
$u \pi \in E(M)$.  By Lallement's Lemma \cite[Lemma 2.4.3]{HoBook}, there exists $e \in E(\fim(X))$ with
$e \pi = u \pi$, and so $u \in K(\P)$.

Conversely, suppose that $K=K(\P)$ is closed.  Then by Proposition \ref{comp_KandN}, we have $K= \nobra{N(\P)}$.
As in section \ref{cliss}, the relation $u \equiv_K v \Longleftrightarrow uv^{-1} \in K$ is a congruence on $\fim(X)$
with kernel $K$, and the quotient $\fim(X) / \! \equiv_K$ is isomorphic to the quotient group $F(X)/N(\P) \sigma_*$,
where $F(X)$ is the free group on $X$ and $\sigma_*$ is the canonical map $\fim(X) \ra F(X)$.  This quotient group is the
maximal group image $\widehat{M}$ of $M$.  By Proposition \ref{functor} there are surjective star-injective functors
$\fim(X) \twobar K \ra  \fim(X) / \! \equiv_K$ and $\fim(X) \twobar K \ra M$ making the square
\[ \xymatrixcolsep{3pc}
\xymatrix{
\fim(X) \twobar K \ar[d]  \ar[r] & M  \ar[d]^{\sigma_*}\\
\fim(X) / \! \equiv_K  \ar[r]_(.65){\cong}  & \widehat{M}}
\]
commute.  It follows that $\sigma_* : M \ra \widehat{M}$ must also be star-injective, and this is equivalent
to $M$ being $E$-unitary (see, for example, \cite[Theorem 2.4.6]{LwBook}).
\end{proof}

\begin{example}  \leavevmode \label{K_not_closed}
\begin{enumerate}
\item Take $M = \I_2$ with $\tau = \mat{2}{1}$ and $\varep = \mat{1}{\ast}$.
Let $X = \{ t,e \}$, and let $\P$ be a presentation of $\I_2$ with generating set $X$, with $t\pi=\tau$ and $e \pi = \varep$.
Since $\I_2$ is not $E$--unitary, $K$ is not closed.  Indeed, $(e^{-1}ete)\pi = 0$ but $(te)\pi = \mat{2}{\ast}$,
and so $e^{-1}ete \in K(\P)$ but $te \not\in K(\P)$.  In consequence, $N(\P)$ here is not closed.
\item We note that the closure $\nobra{N}$ of $N = N(\P)$ does not determine $\simeq _{\P}$.  Consider the free inverse
monoid $M$ on two commuting generators \cite{McAMcF}, presented by $\P = \< a,b : ab=ba \>$.  Then $baba^{-1}b^{-1}b^{-1} \in N$
and so $u = baba^{-1}b^{-1}$ and $v=b$ lie in the same coset of $\nobra{N}$ in $\fim(X)$, but $u \ne v$ in $M$.  This is verified
by mapping $M \ra \I_2$ by
\[ a \mapsto \mat{1}{\ast} \; \text{and} \; b \mapsto \mat{\ast}{2} \,.\]
Then $u$ maps to $0$ but $v$ does not.  We note that in this case, $M$ is $E$--unitary \cite[Proposition 2.4]{McAMcF} , and so $K(\P) = \nobra{N}$.
\end{enumerate}
\end{example}

\end{document}